\documentclass[11pt,a4paper]{amsart}
\usepackage[T1]{fontenc}
\usepackage{textcomp}
\usepackage[margin=1.2in]{geometry}
\usepackage{amssymb}

\usepackage[?]{amsrefs}

\usepackage{xspace}

\theoremstyle{plain}
\newtheorem{thm}{Theorem}[section]
\newtheorem*{thm*}{Theorem}
\newtheorem*{lem*}{Lemma}
\newtheorem{cor}[thm]{Corollary}
\newtheorem{prop}[thm]{Proposition}
\newtheorem{lem}[thm]{Lemma}

\theoremstyle{definition}
\newtheorem{rmk}[thm]{Remark}%[section]

\newcommand{\ov}{\overline}
\newcommand{\eps}{\varepsilon}
\newcommand{\tto}{\longrightarrow}

\DeclareMathOperator{\dist}{dist}

\DeclareMathOperator{\Sz}{Sz}

\newcommand{\HB}{\text{H{\kern -0.35em}B}}

\include{draft}   %%     REMOVE THIS     %%
\begin{document}
\title{Asymptotic geometry of Banach spaces and uniform quotient maps}

\author{S.~J. Dilworth}

\address{Department of Mathematics, University of South Carolina,
Columbia, SC 29208, USA.} \email{dilworth@math.sc.edu}

\author{Denka Kutzarova}
\address{Institute of Mathematics, Bulgarian Academy of
Sciences, Sofia, Bulgaria.} \curraddr{Department of Mathematics, University of Illinois
at Urbana-Champaign, Urbana, IL 61801, USA.} \email{denka@math.uiuc.edu}

\author{G. Lancien}

\address{Universit\'{e} de Franche-Comt\'{e}, Laboratoire de Math\'{e}matiques UMR 6623,
16 route de Gray, 25030 Besan\c{c}on Cedex, FRANCE.}
\email{gilles.lancien@univ-fcomte.fr}

\author{N.~L. Randrianarivony}

\address{Department of Mathematics and Computer Science, Saint Louis University, St. Louis, MO 63103, USA.}
\email{nrandria@slu.edu}
%\date{\today}

\begin{abstract} Recently, Lima and Randrianarivony pointed out the role of the property $(\beta)$ of Rolewicz in nonlinear quotient problems, and answered a ten-year-old question of Bates, Johnson, Lindenstrauss, Preiss and Schechtman.  In the present paper, we prove that the modulus of asymptotic uniform smoothness of the range space of a uniform quotient map can  be compared with the modulus of $(\beta)$ of the domain space.  We also provide conditions under which this comparison can be improved.
\end{abstract}
\subjclass[2010]{Primary 46B80; Secondary 46B20}
\thanks{The first author was partially supported by NSF grant  DMS1101490.    All authors were supported by the Workshop in Analysis and Probability at Texas A\&M University in summer 2011.  The last author was supported in part by a Young Investigator award from this NSF funded Workshop.}% in Analysis and Probability at Texas A\&M University.}

\maketitle

\section{Introduction}

The following definitions are taken from \cite[Chapter 11]{BenyaminiLindenstrauss2000}.  A map $T:X\tto Y$ between two metric spaces $X$ and $Y$ is called \emph{co-uniformly continuous} if for every $\eps >0$, there exists $\delta>0$ such that for every $x\in X$,
$$B(Tx, \delta) \subset T\left(B(x,\eps)\right).$$
If there exists a constant $c>0$ independent of $\eps$ such that $\delta$ can be chosen to be $c.\eps$, then $T$ is called \emph{co-Lipschitz}. Here and throughout this article $B(x,\eps)$ denotes the closed ball of center $x$ and radius $\eps$.

A map $T:X\tto Y$ that is both uniformly continuous and co-uniformly continuous (resp. Lipschitz and co-Lipschitz) is called a \emph{uniform quotient} (resp. \emph{Lipschitz quotient}) map.  If $T$ is also surjective, then $Y$ is called a \emph{uniform quotient}  (resp. \emph{Lipschitz quotient}) of $X$.

\indent

In \cite{LimaRandrianarivony2010}, the use of property $(\beta)$ of Rolewicz (see \cite{Rolewicz1987too} and \cite{Kutzarova1991}) was implemented in the study of uniform quotient maps.  There the authors prove that a Banach space that is a uniform quotient of $\ell_p$ for $1<p<2$ has to be linearly isomorphic to a linear quotient of $\ell_p$.  Another result from \cite{LimaRandrianarivony2010} also states that the Banach space $c_0$ cannot be a uniform quotient of a Banach space with property $(\beta)$.  %These results extend the results of Bates, Johnson, Lindenstrauss, Preiss, and Schechtman in \cite{BatesJohnsonLindenstraussPreissSchechtman1999} which use UAAP to study uniform quotient mappings.

In the present paper, we deepen the techniques used in \cite{LimaRandrianarivony2010} in terms of the asymptotic geometry of the Banach spaces $X$ and $Y$, when $Y$ is a uniform quotient of $X$.

\section{Asymptotic geometry}

The study of the asymptotic geometry of a Banach space goes back to Milman \cite{Milman1971}, where he introduces among other things the notions of asymptotic uniform convexity and asymptotic uniform smoothness, although he uses different names and different notation.  The modern names and notation are introduced in \cite{JohnsonLindenstraussPreissSchechtman2002}.

Let $(X,\|\ \|)$ be a Banach space and
$t>0$. We denote by $B_X$ the closed unit ball of $X$ and by $S_X$
its unit sphere. For $x\in S_X$ and $Y$ a closed linear subspace
of $X$, we define
$$\overline{\rho}_X(t,x,Y)=\sup_{y\in S_Y}\|x+t y\|-1\ \ \ \ {\rm and}\ \ \
\ \overline{\delta}_X(t,x,Y)=\inf_{y\in S_Y}\|x+t y\|-1,$$ then
$$\overline{\rho}_X(t,x)=\inf_{{\rm
dim}(X/Y)<\infty}\overline{\rho}_X(t,x,Y)\ \ \ \ {\rm and}\ \ \ \
\overline{\delta}_X(t,x)=\sup_{{\rm
dim}(X/Y)<\infty}\overline{\delta}_X(t,x,Y).
$$
Finally
$$\overline{\rho}_X(t)=\sup_{x\in S_X}\ \inf_{{\rm
dim}(X/Y)<\infty}\overline{\rho}_X(t,x,Y)\ \ \ \ {\rm and}\ \ \ \
\overline{\delta}_X(t)=\inf_{x\in S_X}\ \sup_{{\rm
dim}(X/Y)<\infty}\overline{\delta}_X(t,x,Y).$$ The norm $\|\ \|$ is said to be
\emph{asymptotically uniformly smooth} (in short AUS) if
$$\lim_{t \to 0}\frac{\overline{\rho}_X(t)}{t}=0.$$
It is said to be \emph{asymptotically uniformly convex} (in short
AUC) if
$$\forall t>0\ \ \ \ \overline{\delta}_X(t)>0.$$

\indent

It is easy to check (see also \cite{JohnsonLindenstraussPreissSchechtman2002}) that a Banach space $X$ that is an $\ell_p$-sum of finite-dimensional spaces with $1<p<\infty$ is both asymptotically uniformly convex and asymptotically uniformly smooth, and that $\ov{\delta}_X(t)=\ov{\rho}_X(t)=(1+t^p)^{1/p}-1$.  Similarly, a Banach space $E$ that is a $c_0$-sum of finite-dimensional spaces is also asymptotically uniformly smooth, and $\ov{\rho}_E(t)=0$ for all $0<t\le 1$. Note that if there exists $\tau>0$ such that $\ov{\rho}_X(t)=0$ for all $0<t<\tau$, the Banach space $X$ is called \emph{asymptotically uniformly flat} (see\cite{GodefroyKaltonLancien2000}).

\indent

Another asymptotic property will be crucial in this paper, as it was in \cite{LimaRandrianarivony2010}. It was introduced by S. Rolewicz in \cite{Rolewicz1987too} and is now called property $(\beta)$ of Rolewicz. For its definition, we shall use a characterization due to D. Kutzarova \cite{Kutzarova1991}.

An infinite-dimensional Banach space $X$ is said to have property $(\beta)$ if for any $t\in (0,a]$, where the number $1\leq a\leq 2$ depends on the space X,  there exists $\delta>0$ such that for any $x$ in $B_X$ and any $t$-separated sequence $(x_n)_{n=1}^\infty$ in $B_X$, there exists $n\ge 1$ so that
$$\frac{\|x-x_n\|}{2}\le 1-\delta.$$
For a given $t\in (0,a]$, we denote $\ov{\beta}_X(t)$ the supremum of all $\delta\ge 0$ so that the above property is satisfied.

\noindent Let us mention that this modulus was computed for $\ell_p$ ($1<p<\infty$) by Ayerbe, Dom\'inguez Benavidez, and Cutillas in \cite{AyerbeDominguez_BenavidesCutillas1994} and was denoted $R^{''}_X(\cdot)$. We have chosen to use here the notation $\ov{\beta}_X(\cdot)$ as we believe it is more informative and more coherent with the other moduli.

It was shown  in \cite{Kutzarova1989} and \cite{MontesinosTorregrosa1992} that there are Banach spaces with property ($\beta$) that are not superreflexive. We give an isomorphic characterization of spaces which can be renormed to have property ($\beta$) below.

\section{Asymptotic geometry and uniform quotient maps}\label{steve1}

We start this section by recalling some elementary facts about uniform quotients that can be found in \cite[Chapter 11]{BenyaminiLindenstrauss2000}.

Let $(X,d_X)$ and $(Y,d_Y)$ be metric spaces and assume that $T: X \tto Y$ is a surjective uniform quotient. Denote by $\Omega^T$ the modulus of continuity of $T$, i.e.
$$\Omega^T(d)=\sup\{d_Y(Tx,Ty): x,y\in X, d_X(x,y)\leq d\}.$$
Note that if $X$ is metrically convex (in particular if $X$ is a normed space), then
$$\left(d>0,\ x,x' \in X,\ d_X(x,x')\ge d\right)\Rightarrow \left(d_Y(Tx,Tx')\le 2\frac{\Omega^T(d)}{d}d_X(x,x')\right).$$
Hence, for any $d>0$, the map $T$ is Lipschitz for large distances with $L_d^T\leq 2\frac{\Omega^T(d)}{d}$, where $$L_d^T=\sup\left\{\frac{d_Y(Tx,Ty)}{d_X(x,y)}: x,y\in X, d_X(x,y)\geq d\right\}$$ is the Lipschitz constant of $T$ for distances $\geq d$.

\noindent Note also that if $Y$ is metrically convex, then $T$ is co-Lipschitz for large distances. It means that for any $d>0$ there exists $c>0$ such that for any $r\ge d$ and any $x\in X$, $B(Tx,cr)\subset T(B(x,r))$.

We now follow the notation from \cite{LimaRandrianarivony2010} and, for $d>0$, we denote by $c_d$ (or $c_d^T$ to avoid ambiguity) the supremum of all $c>0$ such that whenever $r\geq d$, $x\in X$, $d_Y(y,Tx)<cr$, we have $y\in T(B(x,r))$. Note that $\{c_d\}_{d>0}$ is non decreasing. If we assume moreover that $T$  is Lipschitz for large distances and $Y$ unbounded, we have that $C^T:=\lim_{d\to\infty} c_d \leq L_1^T<\infty$ (see \cite[Chapter 11]{BenyaminiLindenstrauss2000} and \cite{LimaRandrianarivony2010} for details).

\medskip We now state our main result.

\begin{thm}{\label{basis}} Let $X$ and $Y$ be infinite-dimensional Banach spaces, $S$ be a subset of $X$, and $T$ be a uniform quotient map from $S$ onto $Y$ (and therefore co-Lipschitz for large distances) which is Lipschitz for large distances.  Let $t_0\in (0,1]$. Assume that there exists $\alpha\in [0,1]$ such that for all $y\in S_Y$ and all $\eta >0$ there exists a normalized basic sequence $(e_n)_n$ with basis constant at most $2$ such that $\|y\pm t_0e_n\|\leq 1+\alpha+\eta$ for every $n\geq 1$. Then
$$\frac 23 \ov{\beta}_X \left( \frac{C^Tt_0}{12 L_1^T}\right) \leq \alpha.$$

\end{thm}

\indent

Note that by the introductory remark at the beginning of this section, this theorem includes the case when $S$ is a linear subspace of $X$ and $T$ is a surjective uniform quotient map.  And of course it also evidently includes the case when $T$ is a Lipschitz quotient map from any subset of $X$ onto $Y$.

\begin{proof} We shall follow the proof and notation of \cite[Theorem 4.3]{LimaRandrianarivony2010}. We present the details of this generalization for the sake of completeness.

Again, let $T:S \tto Y$ be as in the assumptions of the theorem, and let $L_d^T$ be the Lipschitz constant of $T$ for distances $\geq d$.

Fix $0<\eps<1$, $\eta>0$ and pick $d_0>0$ large enough so that
\begin{equation}{\label{eq1}}
C^T-\eps<c_{d_0/3}\leq C^T.
\end{equation}
Since $C^T+\eps>c_{d_0}$, there exist $z_{\eps}\in S$, $R\geq d_0$, and $y_{\eps}\in Y$ such that
\begin{equation}{\label{eq2}}
D:=\|y_{\eps}-Tz_{\eps}\|<(C^T+\eps)R,
\end{equation}
but
\begin{equation}{\label{eq3}}
\|x-z_{\eps}\|>R \text { for all } x\in S \text{ with } Tx=y_{\eps}.
\end{equation}
Note that since $R\geq d_0$ and $y_{\eps}\not \in T\left( B(z_{\eps},R)\right)$, it follows that
\begin{equation}{\label{eq4}}
D=\|Tz_{\eps}-y_{\eps}\|\geq c_{d_0}R\ge (C^T-\eps)R.
\end{equation}

\indent

Divide the line segment $[y_{\eps},Tz_{\eps}]$ into three equal-length subsegments: $[y_{\eps},M]$, $[M,m]$, and $[m,Tz_{\eps}]$.  Then by (\ref{eq2}) and (\ref{eq4}), we have
\begin{equation*}
c_{d_0}\frac{R}{3}\leq \|Tz_{\eps}-m\|=\|m-M\|=\|y_{\eps}-M\|=\frac D3 <(C^T+\eps)\frac R3.
\end{equation*}
This gives
$$\|Tz_{\eps}-m\|=\frac D3<(C^T+\eps)\frac R3=\rho_{\eps}c_{d_0/3},$$
where
\begin{equation}{\label{box1}}
\rho_{\eps}:=\left(\frac{C^T+\eps}{c_{d_0/3}} \right)\frac R3>\frac R3\geq \frac{d_0}{3}.
\end{equation}
Hence, by definition of $c_{d_0/3}$, $m=Tz$ for some $z$ with $\|z-z_{\eps}\|\leq \rho_{\eps}$.  By translation if needed [i.e replacing $T:S \tto Y$ with $\widetilde{T}:\widetilde{S}\tto Y$ where $\widetilde{S}=S-z$, and $\widetilde{T}(x)=T(x+z)-m$], let us assume without loss of generality that $m=0$ and $z=0$.  Then $y_{\eps}=2M$, $\|M\|=\frac D3$, and
\begin{equation}{\label{star}}
\|z_{\eps}\|\leq \rho_{\eps}.
\end{equation}

\indent

We now use the other assumption in our statement and pick a normalized basic sequence $(e_n)_n$ with basis constant at most equal to $2$ such that for all $n\geq 1$,
$$\left \|M\pm t_0\frac D3 e_n\right \|\leq \frac D3 \left(1+\alpha+\eta \right).$$
Set $\displaystyle y_n=M+t_0 \frac D3 e_n$.  Then
$$\|y_n\|\leq \frac D3 \left(1+\alpha+\eta \right)<c_{d_0/3}\left( 1+\alpha+\eta\right) \rho_{\eps}.$$
Since $\rho_{\eps}\geq \frac{d_0}{3}$, it follows that $y_n=Tz_n$, where $\displaystyle \|z_n\|\leq \left(1+\alpha+\eta \right)\rho_{\eps}$.  (Recall that $T(0)=0$.)

\indent

Now,

$$y_n-y_{\eps}=\left(M+t_0\frac D3 e_n\right)-2M=-M+t_0\frac D3 e_n.$$
So,
\begin{equation*}
\begin{split}
\|y_n-y_{\eps}\| &\leq \frac D3 \left(1+\alpha+\eta \right)< c_{d_0/3}\left(1+\alpha+\eta \right) \rho_{\eps}.\\
\end{split}
\end{equation*}
So, as above, $y_{\eps}=Tx_n$, where
$\|z_n-x_n\|\leq \left(1+\alpha+\eta\right) \rho_{\eps}.$

\noindent But since $Tx_n=y_{\eps}$, we also have $\|x_n-z_{\eps}\|>R$ from condition (\ref{eq3}).

\indent

In summary, we have:
\begin{itemize}
\item $\|z_{\eps}\|\leq \rho_{\eps}$,
\item $\|z_n\|\leq \rho_{\eps} \left( 1+\alpha+\eta\right ),$
\item $\|z_n-x_n\|\leq \rho_{\eps}\left(1+\alpha+\eta \right ),$
\item and $\|x_n-z_{\eps}\|>R$.
\end{itemize}

\indent

By the triangle inequality,
\begin{equation*}
\begin{split}
\|z_{\eps}-z_n\|& \geq \|z_{\eps}-x_n\|-\|z_n-x_n\|\\
 &> R-\left(1+\alpha+\eta \right) \rho_{\eps}\\
  &= \left( \frac{3c_{d_0/3}}{C^T+\eps}-1-\alpha-\eta\right)\rho_{\eps} \,\,\,\,\,\,\,\,\,\,\,\,\text{ by (\ref{box1})}\\
  &\geq \left( \frac{3(C^T-\eps)}{C^T+\eps}-1-\alpha-\eta\right)\rho_{\eps}.\\
\end{split}
\end{equation*}

We may assume that $\eps>0$ is sufficiently small that
$$\frac{3(C^T-\eps)}{C^T+\eps}-1>2-\eta.$$
Hence
\begin{equation}
\|z_{\eps}-z_n\|\geq (2-\alpha-2\eta)\rho_{\eps}.
\end{equation}

\indent

On the other hand, since $(e_n)_n$ has basis constant at most $2$, we have
$$\|y_n-y_m\|=\frac{t_0D}{3}\|e_n-e_m\|\geq \frac{t_0D}{6} \text{ whenever }m\neq n.$$
But note that
\begin{equation*}
\begin{split}
\frac{t_0D}{6} &\geq \frac{c_{d_0}}{6}Rt_0 \,\,\,\,\,\,\,\,\,\,\,\,\text{ by (\ref{eq4})}\\
 & \geq \frac{(C^T-\eps)}{6}d_0t_0 \,\,\,\,\,\,\,\,\,\,\,\,\text{ since } R\geq d_0.\\
\end{split}
\end{equation*}
So if we started with $\eps>0$ (sufficiently small), and $d_0<\infty$ (sufficiently large) so that
$$(C^T-\eps)\frac{d_0t_0}{6}>\Omega^T(1),$$
then
$$\Omega^T(1)<\|y_n-y_m\|=\|Tz_n-Tz_m\| \text{ for all } m\neq n.$$
As a result, by definition of $\Omega^T(1)$, we have $\|z_n-z_m\|\geq 1$.  Now apply the Lipschitz constant for distances larger than or equal to $1$ to get
\begin{equation*}
\begin{split}
\|z_n-z_m\| &\geq \frac{\|y_n-y_m\|}{L_1^T}\\
 & \geq \frac{t_0D}{6L_1^T}\\
  &\geq \frac{t_0}{6L_1^T}c_{d_0}R\\
  &=  \frac{t_0}{6L_1^T}\left(\frac{3c_{d_0}c_{d_0/3}}{C^T+\eps}\right)\rho_{\eps}\\
  &\geq \frac{t_0}{2L_1^T}\left(\frac{(C^T-\eps)^2}{C^T+\eps}\right)\rho_{\eps}.\\
\end{split}
\end{equation*}
We may assume $\eps>0$ is sufficiently small that $\displaystyle \frac{(C^T-\eps)^2}{C^T+\eps}\geq \frac{C^T}{2}$, so
$$\|z_n-z_m\|\geq \frac{C^Tt_0\rho_{\eps}}{4L_1^T}.$$

\indent

Recall that $\max(\|z_{\eps}\|,\|z_n\|)\leq \rho_{\eps}\left( 1+\alpha+\eta\right)$, and $\|z_{\eps}-z_n\|\geq \left(2-\alpha-2\eta\right) \rho_{\eps}$.
Notice that
$$\frac{2-\alpha-2\eta}{1+\alpha+\eta}\geq 2-3\alpha-4\eta.$$
Hence by definition of $\ov{\beta}_X$,
\begin{equation}{\label{punchline}}
\ov{\beta}_X\left( \frac{C^Tt_0}{12L_1^T}\right)\leq \ov{\beta}_X\left( \frac{C^Tt_0}{4L_1^T\left(1+\alpha+\eta\right)}\right)\leq \frac 32 \alpha+2\eta.
\end{equation}
Note that $\alpha\leq 1$, and $\ov{\beta}_X$ is nondecreasing.  Also note that
$$\displaystyle \frac{C^Tt_0}{12L_1^T} \,\,\text{ and }\,\, \displaystyle \frac{C^Tt_0}{4L_1^T\left(1+\alpha+\eta\right)}$$
both belong to the domain of $\ov{\beta}_X$ because they are both less than or equal to $\displaystyle \frac{C^Tt_0}{4L_1^T}$, which is less than 1 since $t_0\leq 1$. We finish the proof by noting that $\eta$ is arbitrary.

\end{proof}

\medskip

We can now deduce the following.

\begin{cor}{\label{betaAUS}}
Let $T: S\tto Y$ be a uniform quotient mapping from a subset $S\subset X$ onto $Y$ which is Lipschitz for large distances, where $X$ and $Y$ are infinite-dimensional Banach spaces.  Then
$$\frac 23 \ov{\beta}_X \left( \frac{C^Tt}{12L_1^T}\right)\leq \ov{\rho}_Y(t)$$
for all $0<t\leq 1$.
\end{cor}

\begin{proof} Let $t\in (0,1]$, $y$ in $S_Y$ and $\eta>0$. By definition of $\ov{\rho}_Y(t)$, there exists a finite co-dimensional subspace $Z$ of $Y$ so that $\sup_{z\in S_Z} \|y+tz\|<1+\ov{\rho}_Y(t)+\eta.$
Since $Y$ is infinite-dimensional, so is $Z$ and Mazur's Lemma insures that there exists a 2-basic sequence $(e_n)$ in the unit sphere of $Z$. It follows that  $\|y\pm te_n\|\leq 1+\ov{\rho}_Y(t)+\eta$ for every $n\geq 1$. The conclusion follows from a direct application of Theorem \ref{basis}.  Note that $0\leq \ov{\rho}_Y(t)\leq t\leq 1$.

\end{proof}

\section{Uniform quotient maps and projections}

Let us start with an elementary lemma.

\begin{lem}{\label{lem2}}
Let $T$ be a uniform quotient map from a subset $S$ of a Banach space $X$ onto a Banach space $Y$ which is Lipschitz for large distances.  Let $P$ be a bounded linear projection from $Y$ onto a closed subspace $Y_0$ of $Y$.  Then,
\begin{enumerate}
\item[(a)] $L_d^{PT}\leq \|P\|L_d^T$ for every $d>0$,

\mbox{}
\item[(b)] $c_d^{PT}\geq c_d^T$ for every $d>0$,

\mbox{}
\item[(c)] and $C^{PT}\geq C^T$.

\mbox{}
\end{enumerate}
\end{lem}

\begin{proof}
\mbox{}
\begin{enumerate}
\item[(a)] is clear.

\mbox{}

\item[(b)] Suppose $x\in S$ and $y_0 \in Y_0$ satisfy $\|y_0-PTx\|<c_d^Tr$ where $r\geq d>0.$  Then $$(y_0+(I-P)Tx)-Tx=y_0-PTx.$$
So there exists $x'\in S$ with $\|x-x'\|\leq r$ such that $Tx'=y_0+(I-P)Tx$.  Hence
$$PTx'= Py_0+P(I-P)Tx=y_0.$$
Thus $c_d^{PT}\geq c_d^T$ by definition of $c_d^{PT}$.

\mbox{}

\item[(c)] $\displaystyle C^{PT}=\lim_{d\to \infty} c_d^{PT}\geq \lim_{d\to \infty} c_d^T=C^T$.

\mbox{}
\end{enumerate}
\end{proof}

\indent

\begin{cor}{\label{cor1}}
Assume $X$, $Y$, and $Y_0$ are all infinite-dimensional, let $S\subset X$ and let $T$ and $P$ be maps as in Lemma \ref{lem2}.  Then
$$\frac 23 \ov{\beta}_X\left(\frac{C^Tt}{12\|P\|L_1^T}\right)\leq \ov{\rho}_{Y_0}(t) \text{ for all } 0<t\leq 1.$$
\end{cor}

    \begin{proof}
    Apply Corollary \ref{betaAUS}  to the map $PT$.
    \end{proof}

\indent
    As an immediate consequence we obtain the following.

\begin{cor}{\label{cor2}} Let $(q_n)$ be a sequence in $[1,\infty)$ such that $\lim q_n=\infty$. Suppose that the Banach space $X$ has an equivalent norm with property $(\beta)$.  Then there is no uniform quotient map from any subset of $X$ onto $Y=\left(\displaystyle \sum_{n=1}^{\infty} \ell_{q_n}\right)_{\ell_2}$ that is Lipschitz for large distances.
\end{cor}

 \begin{rmk} In \cite{LimaRandrianarivony2010} it was proved that $c_0$ cannot be a uniform quotient of (or a Lipschitz quotient of a subset of) a Banach space with property $(\beta)$.  Note that a uniformly convex Banach space has property $(\beta)$.  In \cite{BatesJohnsonLindenstraussPreissSchechtman1999} it was shown using the UAAP method that a Banach space that is a uniform quotient of a superreflexive space has to be reflexive.  Hence, both the method in \cite{LimaRandrianarivony2010} and the UAAP method in \cite{BatesJohnsonLindenstraussPreissSchechtman1999} show that $c_0$ cannot be a uniform quotient of a uniformly convex Banach space.  However, there are spaces (see Proposition \ref{prop1} below) with property $(\beta)$ that are not superreflexive. So the UAAP method cannot be used in the same way as in \cite{BatesJohnsonLindenstraussPreissSchechtman1999} to show that $Y$ cannot be a uniform quotient of a Banach space with property $(\beta)$.
 \end{rmk}

\indent The following result is a complement to Corollary \ref{cor2}.

\begin{cor}
Suppose $X$ is isomorphic to a Banach space with property $(\beta)$.  Then there is no uniform quotient map that is Lipschitz at large distances from a subset of $X$ onto $Y=\left(\displaystyle \sum_{n\geq 1} \ell_{p_n} \right)_{\ell_2}$ if $p_n \downarrow 1$ as $n\to \infty$.
\end{cor}

\begin{proof} Suppose that $U$ is a uniform quotient from a subset $S$ of $X$ onto $Y$ which is Lipschitz for large distances. M. Ribe proved in \cite{Ribe1984} that there is a uniform homeomorphism $V$ from $Y$ onto $Z=Y\oplus \ell_1$. Note that $V$ is Lipschitz for large distances. Then $VU$ is a uniform quotient from $S$ onto $Z$ that is Lipschitz for large distances. Since $\ell_1$ is complemented in $Z$, there is a linear quotient $Q$ from $Z$ onto $c_0$. Finally $QVU$ is a uniform quotient from $S$ onto $c_0$ that is Lipschitz for large distances. This is in contradiction with the fact that $X$ has property $(\beta)$ (see Corollary \ref{betaAUS}).

\end{proof}

Note that $\left(\displaystyle \sum \ell_{p_n} \right)_{\ell_2}$ is one-complemented in $\left(\displaystyle \sum L_{p_n} \right)_{\ell_2}$. So it follows trivially from the preceding result that there is no uniform quotient that is Lipschitz for large distances from a subset of a Banach space with property $(\beta)$ onto $\left(\displaystyle \sum L_{p_n} \right)_{\ell_2}$ if $p_n \downarrow 1$ as $n\to \infty$. Let us however indicate an alternate simple proof of this fact which does not use Ribe's deep theorem.

\begin{proof} We start with the following lemma.

\begin{lem}{\label{lem3}}
Let $1<p<\infty$. Then for all $x\in L_p=L_p[0,1]$ with $\|x\|_p=1$, and for all $\eta>0$, there exists a monotone basic sequence $(e_n)_{n\geq 1}$ such that $\|x\pm e_n\|_p\leq 2^{1-\frac 1p}+\eta$.
\end{lem}

\begin{proof} It is enough to consider the vectors of the form $x=\displaystyle \sum_{i=0}^Na_ih_i$ with $\|x\|_p=1$, where $(h_i)_{i\geq 0}$ is the Haar basis, which are dense in the unit sphere of $L_p$.  Let $(r_n)_{n\geq 1}$ be the Rademacher sequence.  Then for $i\geq N+1$ we have
$$\|x(1\pm r_i)\|_p=2^{1-\frac 1p}\|x\|_p=2^{1-\frac 1p},$$
    and $(xr_i)_{i\geq N+1}$ is a monotone basic sequence.
    \end{proof}
Assume now that $T$ is a uniform quotient from a subset $S$ of a Banach space $X$ with property $(\beta)$ onto $Y=\left(\displaystyle \sum L_{p_n} \right)_{\ell_2}$ which is Lipschitz for large distances. We conclude our alternate proof by applying Theorem \ref{basis} to the maps $P_nT$, where $P_n$ is the natural projection from $Y$ onto $L_{p_n}$, and using the fact that $2^{1-\frac {1}{p_n}}\to 1$ as $n\to\infty$.
    \end{proof}

\section{The beta modulus of an $\ell_p$-sum of finite-dimensional spaces}

We begin with an extension of the computation of $\ov{\beta}_{\ell_p}$ that was performed in \cite{AyerbeDominguez_BenavidesCutillas1994}.

\begin{prop}{\label{prop1}}
Let $1\leq p<\infty$ and let $X=\left(\displaystyle \sum_{k=1}^{\infty} E_k\right)_{\ell_p}$, where $(E_k)_{k\geq 1}$ is a sequence of finite-dimensional spaces.  Then for $0\leq t \leq 2^{1/p}$ we have:
$$\ov{\beta}_X(t)=1-\frac 12 \left(\left(1+\left(1-\frac{t^p}{2}\right)^{1/p}\right)^p+\frac{t^p}{2}\right)^{1/p}.$$
\end{prop}
    \begin{proof}
    Suppose that $(x_n)$ is a $t$-separated sequence in $B_X$ .  For $k\geq 1$, let $P_k$ be the natural projection of $X$ onto its subspace $E_k$.  By passing to a subsequence we may assume that there exists $y\in B_X$ such that for all $k\geq 1$,
    $$P_k(x_n)\to P_k(y) \text{ as } n\to \infty.$$
    We may also assume that  $(x_n-y)_{n\geq 1}$ is an almost disjoint sequence with respect to the finite-dimensional decomposition $(E_k)$ and that $\lim_{n\to \infty}\|x_n-y\|= 2^{-1/p}\alpha$, with $\alpha\ge t$.  Hence $\|y\|^p\leq 1-\displaystyle \frac{\alpha^p}{2}$.  Then, for any $x\in B_X$
    \begin{equation*}
    \begin{split}
    \limsup \|x_n+x\|^p &=\|y+x\|^p+\limsup\|x_n-y\|^p\\
     &\leq \left(1+\|y\|\right)^p+\frac{\alpha^p}{2}\\
     &\leq \left(1+\left(1-\frac{\alpha^p}{2}\right)^{1/p}\right)^p+\frac{\alpha^p}{2}\\
     &\leq \left(1+\left(1-\frac{t^p}{2}\right)^{1/p}\right)^p+\frac{t^p}{2}.\\
    \end{split}
    \end{equation*}
    Hence
$$\ov{\beta}_X(t)\geq1-\frac 12 \left(\left(1+\left(1-\frac{t^p}{2}\right)^{1/p}\right)^p+\frac{t^p}{2}\right)^{1/p}.$$

    On the other hand, $\ell_p$ embeds isometrically into $X$, so $\ov{\beta}_X(t) \leq \ov{\beta}_{\ell_p}(t)$, which gives the reverse inequality.  (Note that the computation of $\ov{\beta}_{\ell_p}(t)$ was done in \cite{AyerbeDominguez_BenavidesCutillas1994} for the case $1<p<\infty$.  The case $p=1$ is easily checked by considering the vectors $x=e_1$ and $x_n=\left(1-\frac{t}{2}\right)e_1+\frac{t}{2}e_n$, where $(e_n)_n$ is the unit vector basis of $\ell_1$.)
    \end{proof}

\indent

As a consequence of this computation, we note that the result \cite[Theorem 4.1]{LimaRandrianarivony2010} can be stated in terms of sums of finite-dimensional spaces.  Actually, we have the following.

\begin{cor}{\label{g_thm2}}
Let $X$ be a quotient of a subspace of an $\ell_p$-sum of finite-dimensional spaces, where $1<p<\infty$.  Assume a Banach space $Y$ is a uniform quotient of a subset of $X$, where the uniform quotient map is Lipschitz for large distances.  Then $Y$ cannot contain any subspace isomorphic to $\ell_q$ for any $q>p$.
\end{cor}
    \begin{proof} %It follows from the classical result of Johnson and Zippin \cite{JohnsonZippin1974} that we can assume that $X$ is a quotient of an $\ell_p$-sum of finite-dimensional spaces.  Then since the composition of a linear quotient map with a uniform quotient map is again a uniform quotient map, we can assume that  $X$ is an $\ell_p$-sum of finite-dimensional spaces. On the other hand, we will assume without loss of generality that $Y$ contains a subspace $Z$ isometric to $\ell_q$ for some $q>p$.  Hence we note that $\ov{\beta}_X(t)\sim t^p$ and $\ov{\rho}_Z(t)\sim t^q$.

Let $T$ be a uniform quotient map from a subset $S$ of $X$ onto $Y$ that is Lipschitz for large distances (and co-Lipschitz for large distances).  Let $Z$ be a subspace of $Y$ that is isomorphic to $\ell_q$, and call $J:Z\tto \ell_q$ the linear isomorphism.  Consider the further subset $S':=T^{-1}(Z)$, and consider the restriction map $T':=T_{|_{S'}}:S'\tto Z$.  Then $T'$ is a uniform quotient map from $S'$ onto $Z$, and as a restriction map it inherits the Lipschitz for large distances and co-Lipschitz for large distances properties.  Now the composition $J\circ T'$ is a uniform quotient map from $S'$ onto $\ell_q$ that is Lipschitz for large distances (and co-Lipschitz for large distances).

On the other hand, let $Q:W\tto X$ be a linear quotient map, where $W$ is a subspace of an $\ell_p$-sum of finite-dimensional spaces.  Repeating the previous argument, we note that $Q':=Q_{|_{Q^{-1}(S')}}\tto S'$ is a surjective Lipschitz quotient map.

Finally, we note that the composition $J\circ T'\circ Q'$ is a uniform quotient map from the subset $Q^{-1}(S')$ of an $\ell_p$-sum of finite-dimensional spaces onto $\ell_q$ that is both Lipschitz for large distances and co-Lipschitz for large distances.  We apply Corollary  \ref{betaAUS}  and Proposition \ref{prop1} to get a contradiction.

    %We treat the Lipschitz version first, as it is easier.  Let $T$ be a Lipschitz quotient map from a subset $S$ of $X$ onto $Y$.  Consider the further subset $S':=T^{-1}(Z)$ of $S$.  Then $T$ is a Lipschitz quotient map from $S'$ onto $Z$.  We get a contradiction from Corollary \ref{betaAUS}.

    %The uniform version is similar but needs more care.  Let $T$ be a uniform quotient map from $X$ onto $Y$.  Then $T$ is Lipschitz and co-Lipschitz for large distances since both $X$ and $Y$ are metrically convex.  Again, consider the \emph{subset} $S'=T^{-1}(Z)$ of $X$, and consider the restriction map $T':=T_{|_{S'}}:S'\tto Z$.  Then as a restriction map, $T'$ inherits the Lipschitz and co-Lipschitz for large distances properties.  Hence we can compute $\Omega^{T'}(1)$, $c_d^{T'}$, and $C^{T'}$, and the computations in the proof of Theorem \ref{basis} all work out for $T'$.

We thank the anonymous referee for helping us simplify the proof of this corollary.

\end{proof}

Let us now treat a brief example. We recall that a function $F: [0,\infty)  \to [0,\infty)$ is an \emph{Orlicz function} if it is continuous, non decreasing, convex and such that $F(0)=0$ and $\lim_{t\to \infty}F(t)=\infty$. Then the space $\ell_F$ is the space of all real sequences $x=(x_n)_{n=1}^\infty$ such that there exists $r>0$ satisfying
$$\sum_{n=1}^\infty F\left(\frac{|x_n|}{r}\right) <\infty.$$
It is equipped with the Luxemburg norm
$$\forall x\in \ell_F\ \ \|x\|_F=\inf \left\{r>0,\ \sum_{n=1}^\infty F\left(\frac{|x_n|}{r}\right)\le 1\right\}.$$
Then the Orlicz space $h_F$ is the closure in $\ell_F$ of the finitely supported sequences. The $\ell_p$ spaces that are isomorphically contained in $h_F$ can be precisely described with the Matuszewska-Orlicz indices (also called Boyd indices) and defined as follows:
$$\alpha_F=\sup \left\{q,\ \sup_{0<u,v\le1}\frac{F(uv)}{u^qF(v)}<\infty\right\}\ \ {\rm and}\ \
\beta_F=\inf \left\{q,\ \inf_{0<u,v\le1}\frac{F(uv)}{u^qF(v)}>0\right\}.$$
J. Lindenstrauss and L. Tzafriri (see \cite[page 143]{LindenstraussTzafriri1977}) proved that $h_F$ isomorphically contains $\ell_p$ (or $c_0$ if $p=\infty$) if and only if $p\in [\alpha_F,\beta_F]$. Then, the following is an immediate consequence of Corollary \ref{g_thm2}.

\begin{cor}
Let $X$ be a subspace of a quotient of an $\ell_p$-sum of finite-dimensional spaces with $1<p<\infty$ and let $F$ be an Orlicz function.  Assume that $h_F$ is a uniform quotient of (or a Lipschitz quotient of a subset of) $X$.  Then $\beta_F\le p$.
\end{cor}

\section{Isomorphic results}

In this section, we explore further consequences of Theorem \ref{basis}. In particular, we shall try to express our initial results about asymptotic moduli in terms of the associated isomorphic invariants. It is now well known that the asymptotic uniform smoothness is closely related to the Szlenk index. Let us recall its definition.

Let $X$ be a real Banach space and $K$ a
weak$^*$-compact subset of $X^*$. For $\eps>0$ we let $\mathcal V$ be
the set of all relatively weak$^*$-open subsets $V$ of $K$ such
that the norm diameter of $V$ is less than $\eps$ and
$s_{\eps}K=K\setminus \cup\{V:V\in\mathcal V\}.$ Then we define
inductively $s_{\eps}^{\alpha}K$ for any ordinal $\alpha$ by
$s^{\alpha+1}_{\eps}K=s_{\eps}(s_{\eps}^{\alpha}K)$ and
$s^{\alpha}_{\eps}K={\displaystyle
\cap_{\beta<\alpha}}s_{\eps}^{\beta}K$ if $\alpha$ is a limit
ordinal. We then define $\text{Sz}(X,\eps)$ to be the least ordinal $\alpha$
so that $s_{\eps}^{\alpha}B_{X^*}=\emptyset,$ if such an ordinal
exists. Otherwise we write $\text{Sz}(X,\eps)=\infty.$ The \emph{
Szlenk index} of $X$ is finally defined by
$\text{Sz}(X)=\sup_{\eps>0}\text{Sz}(X,\eps)$. In the sequel
$\omega$ will denote the first infinite ordinal.

The seminal result on AUS renormings is due to H. Knaust, E. Odell
and T. Schlumprecht (\cite{KnaustOdellSchlumprecht1999}). Among other things, they proved the following.

\begin{thm}\label{KOS} {\bf (Knaust-Odell-Schlumprecht 1999)}  Let $X$ be a separable Banach space. The following assertions are
equivalent.

(i) $X$ admits an equivalent AUS norm with a power type modulus.

(ii) Sz$(X)\leq \omega$.

(iii) There exist $C>0$ and $p\in [1,+\infty)$ such that: $\forall
\eps>0$ Sz$(X,\eps)\leq C\eps^{-p}$.
\end{thm}

\indent

This was extended by M. Raja in \cite{Raja2010} to the non separable
case.

\indent

Let now $X$ be a Banach space with
$\text{Sz}(X)\leq\omega$ and define the Szlenk power type of $X$
to be $$p_X:=\inf\{q\geq 1,\ \sup_{\eps>0} \eps^q\text{Sz}(X,\eps)<\infty\}.$$

We can now state:

\begin{cor} Let X be a separable Banach space such that for all $t\in [0,1]$, $\ov{\beta}_X(t)\ge ct^p$ (for some $c>0$ and some $p\in (1,\infty)$). Assume that a Banach space $Y$ is a uniform quotient of $X$ (or a Lipschitz quotient of a subset of $X$). Then $p_Y\ge p*$, where $p*$ is the conjugate exponent of $p$.
\end{cor}

\begin{proof} This is a direct consequence of Theorem \ref{betaAUS} and Theorem 4.8 in \cite{GodefroyKaltonLancien2001} which  insures that
$$p_Y=\inf\{q\geq 1,\ \text{there\ is\ an\ equ.\ norm}\ |\ |\ \text{on}\ Y,
\ \ \exists c>0\ \forall t>0,\ \ov{\rho}_{|\ |}(t)\le ct^{q*}\}.$$
\end{proof}

\indent

We now wish to summarize  what is known about the isomorphic characterization of property $(\beta)$. We shall limit ourselves to separable Banach spaces. First, we need to mention some notions closely related to asymptotic uniform smoothness and asymptotic uniform convexity. The first one is the
\emph{uniform Kadec-Klee} property (UKK) introduced by R. Huff \cite{Huff1980}. He also defined the so-called \emph{nearly uniformly convex} (NUC) spaces and proved in \cite{Huff1980} that a Banach space $X$ is NUC if and only if it is reflexive and UKK, or equivalently, reflexive and AUC. Then S. Prus defined in \cite{Prus1989} the \emph{nearly uniformly smooth} (NUS) spaces and showed that a Banach space is NUS if and only if its dual is NUC. This can be rephrased as follows: a reflexive Banach space is AUS if and only if its dual is AUC. Finally, D. Kutzarova proved in \cite{Kutzarova1990} that if $X$ is a Banach space with a Schauder basis, then $X$ admits an equivalent norm with property $(\beta)$ if and only if $X$ admits an equivalent NUC norm and an equivalent NUS norm. One can gather these works to obtain the following statement.

\begin{thm} Let $X$ be a separable Banach space. The following assertions are equivalent.

(i) $X$ admits an equivalent norm with property $(\beta)$.

(ii) $X$ is reflexive, $\text{Sz}(X)\le \omega$ and $\text{Sz}(X^*)\le \omega$.

(iii) $X$ is reflexive, admits an equivalent AUS norm and an equivalent AUC norm.

(iv)  $X$ is reflexive and admits an equivalent norm which is simultaneously AUS of power type $q$ and AUC of power type $p$ for some  $1 < q \le p < \infty$

(v)  $X$ embeds isomorphically into a reflexive Banach space $Y$ with a finite-dimensional decomposition which satifies $1-(p,q)$-estimates, where $p$ and $q$ are  as in (iv).

\end{thm}

\begin{proof} The equivalences $(ii) \Leftrightarrow (iii) \Leftrightarrow (iv)$ can be derived from \cite{KnaustOdellSchlumprecht1999} and are explicit in \cite{OdellSchlumprecht2006} (Theorem 7 and Remark 1) while
$(iv) \Leftrightarrow (v)$ is  contained in Corollary 2.14 of \cite{OdellSchlumprecht2006/2}.  The implication $(iv)\Rightarrow (i)$ follows from Theorem 4 in \cite{Kutzarova1990} insuring that a Banach space which is both NUS and NUC has property $(\beta)$.

\noindent Finally, assume that $X$ has property $(\beta)$. Then it is NUC (see \cite{Rolewicz1987}). On the other hand it is proved in \cite{Kutzarova1990} that if $X$  admits a Schauder basis then it has an equivalent NUS norm (Corollary 8).
Since property $(\beta)$ passes clearly to quotients, any quotient of $X$ with a Schauder basis has an equivalent NUS norm and therefore a Szlenk index at most $\omega$. Then, it follows from \cite{Lancien1993} (Proposition 3.5) that $\text{Sz}(X)\le \omega$. Finally, we can apply Theorem \ref{KOS} to deduce that $X$ admits an equivalent NUS norm. (Let us mention an alternative  proof of the  fact that if $X$ has property $(\beta)$ then $X$ has an equivalent NUS norm. The argument of Theorem 7 of \cite{Kutzarova1990} carries over to  show that any shrinking Markushevitch basis can be blocked to satisfy $(\infty, q)$-estimates. Combining this with Theorem 4.1 of \cite{Prus1989}, and the observation made at the end  of \cite{Prus1989} that the proof of  Theorem 4.1 carries over to biorthogonal systems, gives the result.)

\end{proof}

\begin{rmk} It is proved in \cite{BaudierKaltonLancien2010} that this class of Banach spaces: the ``reflexive spaces $X$ with  $\text{Sz}(X)\le \omega$ and $\text{Sz}(X^*)\le \omega$'' is stable under coarse Lipschitz embeddings and therefore under uniform homeomorphisms.
%We do not know if it is stable under uniform quotients.
\end{rmk}

\indent

The rest of the article is devoted to another application of Theorem \ref{basis}.  A Banach space $X$ has property $(M)$ if for every weakly null sequence $(x_n)$ and every $u,v\in S_X$,
$$\limsup \|u+x_n\|=\limsup \|v+x_n\|.$$
Property $(M)$ was introduced by N. Kalton in \cite{Kalton1993}. It is proved in \cite{Kalton1993} that the Orlicz sequence spaces have an equivalent norm with property $(M)$ and in \cite{AndroulakisCazacuKalton1998} that the same is true for Fenchel-Orlicz spaces. Note that for $1\le p\neq 2<\infty$, $L_p$ does not admit any equivalent norm with property $(M)$ (see \cite{Kalton1993}).  We should also mention that a separable Banach space with property $(M)$ and not containing $\ell_1$ is AUS (see \cite{DuttaGodard2008}).

Here we consider the consequences of property $(M)$ on the asymptotic uniform convexity of a space and for the existence of non linear quotient maps.

\begin{lem}{\label{g_lem1}}
Assume $X^*$ is separable.  Then for all $x\in S_X$ and all $0<t\leq 1$, one has
$$\ov{\delta}_X(t,x)=\inf \left \{\limsup \|x+tx_n\|-1:x_n \to 0 \text{ {\rm weakly}}, \|x_n\|= 1 \right \}.$$
\end{lem}
    \begin{proof}
    Let $x\in S_X$, $0<t\leq 1$, $Y$ a subspace of $X$ such that $\dim X/Y<\infty$ and $x_n\to 0$ weakly with $\|x_n\|= 1$.  Since $(x_n)_n$ is weakly null, we have $\dist(x_n,Y)\to 0$.  So consider $(y_n)_n \subset Y$ such that $\|x_n-y_n\| \to 0$.  Therefore $\|y_n\|\to 1$ and
    \begin{equation}
    \limsup \|x+tx_n\|-1 = \limsup \|x+ty_n\|-1\geq \ov{\delta}_X(t,x,Y).
    \end{equation}
    Since $Y$ is arbitrary, we obtain that
    $$\limsup \|x+tx_n\|-1\geq \ov{\delta}_X(t,x).$$

    Conversely, let $(x^*_n)_{_{n\geq 1}}$ be dense in $X^*$.  Call $Y_n=\displaystyle \bigcap_{i=1}^n \ker x^*_i$.  Let $x\in S_X$, $t>0$ and $\eps>0$.  For each $n$, pick $y_n\in S_{Y_n}$ such that 
$$\|x+ty_n\|-1 \leq \ov{\delta}_X(t,x,Y_n)+\eps.$$
%It follows from the density of $(x_n^*)$ in $X^*$ that $(y_n)$ is weakly null.  On the other hand,
We have $\ov{\delta}_X(t,x)\geq \ov{\delta}_X(t,x,Y_n)$, and therefore
$$\limsup \|x+ty_n\|-1 \le \ov{\delta}_X(t,x)+\eps.$$
On the other hand, the density on $(x_n^*)$ in $X^*$ implies that $(y_n)$ is weakly null.  Since $\eps>0$ is arbitrary, this finishes the proof.
    \end{proof}

The following is then immediate.

\begin{cor}{\label{conseq}}
Assume $X^*$ is separable and $X$ has property $(M)$.  Then for all $u,v\in S_X$,
$$\ov{\delta}_X(t,u)=\ov{\delta}_X(t,v)=\ov{\delta}_X(t).$$
\end{cor}

\begin{lem}{\label{g_lem2}}
Let $Y$ be an infinite-dimensional Banach space, $u\in S_Y$, $t\in (0,1]$ and $\eta>0$.  Then there exists a $2$-basic sequence $(x_n)_n\subset S_Y$ so that for each $n$,
$$\|u+tx_n\|\leq 1+ \ov{\delta}_Y(t,u)+ \frac{\eta}{2}.$$
\end{lem}
    \begin{proof} By definition, for any finite-codimensional subspace $Z$ of $Y$ $\ov{\delta}_Y(t,u,Z)\le \ov{\delta}_Y(t,u)$. In his classical construction, Mazur builds inductively $Y_1=Y$, $x_1\in S_{Y_1}$,...., $Y_{n-1}\supset Y_{n}$ of finite codimension in $Y$, $x_n\in S_{Y_n}$... such that $(x_1,..,x_n,..)$ is a 2-basic sequence. When picking $x_n$ in $S_{Y_n}$, one can also insure that
$$\|u+tx_n\|\le 1+\ov{\delta}_Y(t,u,Y_n)+ \frac{\eta}{2}\le 1+\ov{\delta}_Y(t,u)+ \frac{\eta}{2}.$$

 \end{proof}

\begin{lem}{\label{g_lem2_cont}}
If moreover $Y^*$ is separable and $Y$ has property $(M)$, then for all $u\in S_Y$, all $t\in (0,1]$ and all $\eta>0$, there exists a $2$-basic sequence $(e_n)_n \subset S_Y$ such that for each $n$,
$$\|u\pm te_n\|\leq 1+  \ov{\delta}_Y(t)+\eta.$$
\end{lem}
    \begin{proof} The two previous statements imply that for all $n$, $\|u+tx_n\|\le 1+\ov{\delta}_Y(t)+ \frac{\eta}{2}$. Moreover, the separability of $Y^*$, allows us, by carefully choosing the $Y_n$'s in the proof of Lemma \ref{g_lem2}, to have that $(x_n)$ is weakly null.  Then, it follows from property $(M)$ that $\limsup \|u-tx_n\|=\limsup \|u+tx_n\|.$ We conclude the proof by taking a subsequence $(e_n)$ of $(x_n)$.
    \end{proof}

\indent

We can now prove the following.

\begin{thm}{\label{g_thm1}}
Let $X$ be a Banach space with property $(\beta)$ and $Y$ be a separable Banach space with property $(M)$. Assume that $Y$ is a uniform quotient of a subset of $X$, where the uniform quotient map is Lipschitz for large distances. Then $Y$ is AUS, AUC and satisfies:
$$\exists C\ge 1\ \ \forall t\in (0,1]\ \ \frac 1C \ov{\beta}_X\left(\frac tC \right)\leq \ov{\delta}_Y(t).$$
\end{thm}

\begin{proof} Assume first that $Y$ contains a linear copy of $\ell_1$. It follows from Theorem 2.1 in \cite{JohnsonLindenstraussPreissSchechtman2002b} that every separable Banach space is a Lipschitz quotient of $Y$. In particular, $c_0$ would be a uniform quotient of a subset of $X$ that is Lipschitz for large distances, which is impossible. Therefore $Y$ does not contain $\ell_1$ and, by Proposition 2.2 in \cite{DuttaGodard2008}, $Y$ is AUS. In particular, $Y^*$ is separable and we can apply Lemma \ref{g_lem2_cont} and Theorem \ref{basis} to obtain that
$$\forall t\in (0,1]\ \ \frac 23 \ov{\beta}_X\left(\frac{C^Tt}{12L_1^T}\right)\leq \ov{\delta}_Y(t).$$

\end{proof}

We shall now relate this to a ``Szlenk type'' derivation on the space $Y$ itself. For a weakly-closed subset $F$ of $Y$ and $\eps>0$, we define
    $$\sigma'_{\eps}(F)=F\backslash \bigcup \{\text{all weakly-open subsets of }F\text{ of diameter }\leq \eps\}.$$
Then $\sigma_\eps^\alpha(F)$ is defined inductively as usual for $\alpha$ ordinal and  $S(Y,\eps)=\inf \{\alpha:(B_Y)_{\eps}^{\alpha}=\emptyset \}$ if it exists ($:=\infty$ otherwise).

\begin{cor}{\label{g_cor}}
Under the assumptions of Theorem \ref{g_thm1}, there exists $K\geq 1$ such that for all $\eps \in (0,1)$,
$$S(Y,\eps)\leq K \ov{\beta}_X\left(\frac{\eps}{K}\right)^{-1}.$$
\end{cor}
    \begin{proof} Assume, as we may, that $Y$ has a separable dual and an asymptotically uniformly convex norm.  For any $0<r\leq 1$, we show that
    \begin{equation}{\label{sigma}}
    \sigma'_{\eps}(rB_Y) \subset \left( 1-\frac 12 \ov{\delta}_Y\left(\frac{\eps}{2}\right)\right) rB_Y,
    \end{equation}

    In fact, let $(y^*_n)_n$ be dense in $Y^*$.  Consider $y\in \sigma'_{\eps}(rB_Y)$ and without loss of generality, assume $y\neq 0$.  Note that the weakly open subset of $rB_Y$ given by $\Omega_n:=\{z\in rB_Y: \forall i=1,\cdots  n, |y^*_i(z-y)|<1/n\}$ contains $y$ and hence must have diameter greater than $\eps$.  In particular, one must find an element $y_n\in \Omega_n$ such that $\|y_n-y\|>\eps/2$.  By the density of $(y^*_n)_n$, we get that the sequence $(y_n-y)_n$ is weakly null.  By taking a subsequence, we may assume that $\lim \|y_n-y\|=t\geq \eps/2$.  From the equality
    $$\frac{y_n}{\|y\|}=\frac{y}{\|y\|}+\frac{t}{\|y\|}\cdot \frac{y_n-y}{t}$$
and noting that $\|y_n\|\leq r$, Lemma \ref{g_lem1} gives us $r/\|y\|\geq 1+\ov{\delta}_Y(t/\|y\|)\geq 1+\ov{\delta}_Y(t)\geq 1+\ov{\delta}_Y(\eps/2).$  Hence $\|y\|\leq r\left(1-\frac 12 \ov{\delta}_Y(\eps/2)\right)$.

Let us denote $\delta=\frac 12 \ov{\delta}_Y(\eps/2)$.  Iterating (\ref{sigma}) we get $\sigma_\eps^{(n)}(B_Y) \subset (1-\delta)^n B_Y$, so we can find some $n_0$ such that $\sigma_\eps^{(n_0)}(B_Y) \subset \frac 12 B_Y$.  In fact, a simple calculation shows that we can choose $n_0\leq \frac{3\ln(2)}{\delta}$.

Next, the annulus $B_Y\backslash \frac 12 B_Y$ contains a ball of radius $\frac 14$.  So we have $\sigma_\eps^{(n_0)}(\frac 14 B_Y)=\emptyset$, i.e $\sigma_{4\eps}^{(n_0)}(B_Y)=\emptyset$.  This gives $S(Y,4\eps)\leq \frac{6\ln(2)}{\ov{\delta}_Y(\eps/2)}$, or $S(Y,\eps)\leq \frac{6\ln(2)}{\ov{\delta}_Y(\eps/8)}$.  The conclusion now follows from Theorem \ref{g_thm1}.
\end{proof}

\begin{rmk} In particular if $Y$ is a separable reflexive uniform quotient of a Banach space $X$ with property $(\beta)$ and $Y$ admits an equivalent norm with property $(M)$, then $Y$ has an equivalent norm with property $(\beta)$ and there is a constant $K\geq 1$ such that for all $\eps\in (0,1)$,
$$S(Y,\eps)=\Sz(Y^*,\eps)\leq K \ov{\beta}_X\left(\frac{\eps}{K}\right)^{-1}.$$
%We do not know if this is still true if we do not assume that $Y$ admits an equivalent norm with property (M). We do not know either if a uniform quotient of a space with property $(\beta)$ is always reflexive.
\end{rmk}

\noindent {\bf Aknowledgements.} This work was initiated during the Concentration Week on ``Non-linear geometry of Banach spaces, geometric group theory and differentiability'' organized at Texas A\&M University (College Station) in August 2011. We wish to thank F. Baudier, W.~B. Johnson, P. Nowak and B. Sari for the perfect organization of this meeting.

We also thank the anonymous referee for helping us improve the presentation of this paper.

\end{document}